\documentclass[11pt]{amsart}
\usepackage{amsfonts,amssymb,amsmath,amsthm}
\usepackage{url,mathabx,color}
\usepackage{enumerate,comment}
\usepackage{mathtools}
\usepackage{enumerate}

\def\smallspace{\negthickspace}

\def\defmin{{\fontsize{2.5}{4}\selectfont def}}

\newcommand{\tmatrix}[9]
{\left(\begin{smallmatrix} 
#1 & #2 & #3\\ #4 & #5 & #6 \\ #7 & #8 & #9\end{smallmatrix} \right)}

\title{On the geometry of idempotents in von Neumann algebras}

\author{Thierry Giordano}

\address{Department of Mathematics and Statistics\\
University of Ottawa\\
585 King Edward Avenue\\
Ottawa, Ontario\\
KiN 6N5\\
Canada}

\email{giordano@uottawa.ca}
\author{Adam Sierakowski${}^1$}

\address{School of Mathematics and Applied Statistics \\
Building 39C\\
University of Wollongong\\
Wollongong NSW \\
2522\\
Australia}

\email{asierako@uow.edu.au}
\thanks{${}^1$Communicating author.}

\newcounter{theorem}
\newtheorem{thm}[theorem]{Theorem}
\newtheorem*{thmd*}{Theorem~D\textsuperscript{$\prime$}}
\newtheorem{lem}[theorem]{Lemma}
\newtheorem{prop}[theorem]{Proposition}

\newtheorem{question}[theorem]{Question}

\theoremstyle{remark}
\newtheorem*{remark*}{Remark}
\newtheorem{remark}[theorem]{Remark}
\newtheorem{remarks}[theorem]{Remarks}

\theoremstyle{definition}
\newtheorem{df}[theorem]{Definition}
\newtheorem{ntn}[theorem]{Notation}

\newcommand{\I}{\mathcal{I}}

\renewcommand*{\thetheorem}{\roman{theorem}}

\newcommand{\eqv}{\approx}

\newcommand{\Inv}{Inv}

\newcommand\mydef{\mathrel{\stackrel{\makebox[0pt]{\mbox{\normalfont\tiny\tiny \defmin}}}{=}}}

\def \bib(#1;#2;#3;#4;#5;#6) {{#1}, {\it #2} {#3},
{\bf#4} (#5) {#6}\par\smallskip}

\date{\today}
\subjclass[2010]{46L35, 46L05, 46L80}

\numberwithin{equation}{section}

\begin{document}

\maketitle

\begin{abstract}
We consider the general linear group as an invariant of von Neumann factors. We prove that up to complement, a set consisting of all idempotents generating the same right ideal admits a characterisation in terms of properties of the general linear group of a von Neumann factor. We prove that for two Neumann factors, any bijection of their general linear groups induces a bijection of their idempotents with the following additional property: If two idempotents or their two complements generate the same right ideal, then so does their image. This generalises work on regular rings, such include von Neumann factors of type $I_{n}$, $n < \infty$.
\end{abstract}

\renewcommand*{\thetheorem}{\Alph{theorem}}
\section*{Introduction}
This project is our first contribution in an ongoing classification of von Neumann factors. Here we consider the general linear group as an invariant. We study how the following geometry of idempotents can be characterised by the general linear group. Recall that an idempotent, also called a generalised projection, is an element satisfying ${e^2=e}$. Let $N$ be a von Neumann factor. For each idempotent $e\in N$ let $[e]$ denote the equivalence class of all idempotents $f$ satisfying $fN=eN$. We consider the following question: 

\begin{question}\label{qstA}
Is there a way to characterise the set $[e]$ of idempotents generating the same right ideal in terms of properties of the general linear group of $N$?
\end{question}

When considering the unitary group as an invariant Dye \cite{Dye} proved the following result back in 1954: Let $N$ and $M$ be two von Neumann factors not of type $I_{2n}$ and $\varphi$ a group isomorphism between their unitary groups, then there exist a linear or conjugate linear $^*$-isomorphism of $N$ and $M$ whose restriction to the unitary group  agrees up to character with $\varphi$.

Our results follows closely the work by Baer \cite{MR0052795} and Ehrlich \cite{MR0081885}. Bear considered the finite dimensional case and Ehrlich studied regular rings, these include precisely the von Neumann factors of finite dimension, this was pointed out by von Neumann in \cite{MR0120174}.

Let $N$ be a von Neumann algebra. Recall that there is a canonical bijection $\iota_N$, $e\mapsto 2e-1$ from the set of idempotents $I(N)$ in $N$ into the set of involutions $\Inv(N)$ in $N$. For each $e\in I(N)$ we let $\Delta^+(e)$ denote the image of $[e]$ via $\iota_N$ and $\Delta^-(e)$ the set $-\Delta^+(e)$. Two idempotents generate the same right ideal precisely when their associated $\Delta^+$-sets are equal. Question~\ref{qstA} is therefore equivalent to asking for a characterisation of a $\Delta^+$-set in term of properties of the general linear group $GL(N)$. Our main theorem states as follows:

\begin{thm}\label{thmB}
Let $N$ be a von Neumann factor. Let $\phi$ be a nonempty subset of $\Inv(N)$. Then $\phi$ is a $\Delta^+$-set or a $\Delta^-$-set if and only if $\phi$ is a maximal set among the nonempty subset of involutions in $N$ satisfying \eqref{delta.a}-\eqref{delta.d}:
\begin{enumerate}
\item\label{delta.a} If $u,v,w\in \phi$ then $uvw=wuv\in \phi$.
\item\label{delta.b} If $u,v\in \phi$ then there exist a unique $w\in \phi$ s.t. $wvw=u$.
\item\label{delta.c} If $u\in \Inv(N)$, then $u\phi=\phi u$ iff $uw=wu$ for some $w\in\phi$.
\item\label{delta.d} If $t\in \phi^2$ then $(t-1)^2=0$.
\end{enumerate}
\end{thm}

The final item \emph{(4)} is not stated as a property of $GL(N)$, but this can be done as follows: For each element $t\in N$ let $C(t)$ denote the centraliser of $t$ in $GL(N)$ and $C^2(t)$ the second centraliser of $t$ in $GL(N)$. We have

\begin{thm}\label{thmC}
Let $N$ be a von Neumann factor containing $1\neq s\in GL(N)$. Then $(s-1)^2=0$ if and only if the following \eqref{class.1C}-\eqref{class.4C} holds:
\begin{enumerate}
\item\label{class.1C} If $t\in GL(N)$, then $C(s)\subsetneq C(t)$ iff $GL(N)=C(t)$.
\item\label{class.2C} There exists $u\in Inv(N)$ such that $usu=s^{-1}$, and
\item\label{class.3C} an element $r\in C^2(u)$ such that $rsr^{-1}=s^2$.
\item\label{class.4C} $s^3\neq 1$.
\end{enumerate}
\end{thm}

Having established a characterisation of $\Delta^+$-sets (up to a sign) we apply this result to bijections between the general linear groups of von Neumann factors. For $e, f\in I(N)$, write $e\eqv f$ whenever $eN=fN$. This is an equivalence relation on $I(N)$. Let $[e]\mydef\{f \in I(N): f\eqv e\}$ denote the equivalence class containing $e$ and $I(N)/\smallspace\eqv\ \mydef \{[e] : e\in I(N)\}$ the set of equivalence classes.

\begin{thm}\label{thmD}
Let $N$ and $M$ be two von Neumann factors and let $\varphi$ a group isomorphism of their general linear groups. Let $\theta$ be the bijection of idempotents induced by $\varphi$, i.e., $\theta= \iota^{-1}_M\circ \varphi \circ \iota_N$. Then there exist a partitioning of the nontrivial elements of $I(N)/\smallspace\eqv$ into two set $\I_o,\I_{\bar{o}}$, such that
\begin{align*}
\tilde\theta ([e])=\left\{
\begin{array}{ll}
\theta ([e]),&\mbox{ if } e \in \I_o\\
\theta ([1-e]),&\mbox{ if } e \in \I_{\bar{o}}\\
{[1]},&\mbox{ if } e =1\\
{[0]},&\mbox{ if } e =0
\end{array}\right.
\end{align*}
is a bijection of $I(N)/\smallspace\eqv$ and $I(M)/\smallspace\eqv$.
\end{thm}
It is possible that $\I_o$ or  $\I_{\bar{o}}$ is the empty set, consider for example the bijections $\varphi(u)=u^*$ or $\varphi(u)=u$.

\renewcommand*{\thetheorem}{\roman{theorem}}
\numberwithin{theorem}{section}
\section{Elements of class $2$ and the proof of Theorem~\ref{thmC}.}
In this section we give a characterisation of elements of class 2 (see Definition~\ref{def.classtwo}) in terms of properties of the general linear group. We start with a few definitions. Let $N$ be a von Neumann algebra. We let  $GL(N)$ denote the set of \emph{invertible} elements in $N$, $I(N)$ the set of \emph{idempotents} in $N$, $\Inv(N)$ the set of \emph{involutions} in $N$, and $Z(N)$ the \emph{center} of $N$, i.e.,
\begin{align*}
&GL(N)\mydef\{u\in N : uv=vu=1 \ \text{for some}\ v\in N \}, \ \  I(N)\mydef \{e\in N : e^2=e\}, \\
&\Inv(N)\mydef\{u\in N : u^2=1\}, \ \ \ Z(N)\mydef \{x\in N :  xy=yx \ \text{for all}\ y\in N\}.
\end{align*}
The inverse of $u\in GL(N)$ is unique and is denoted $u^{-1}$. Two idempotents are \emph{orthogonal} if they commute and their product is zero. For $u\in N$ set
\begin{align*}
&C(u)\mydef \{t \in GL(N) : ut=tu\},\\
&C^2(u)\mydef \{v \in GL(N) : tv=vt \ \text{for all}\ t\in C(u)\}.
\end{align*}

\begin{df}[cf.~\cite{MR0052795, MR0081885}]\label{def.classtwo}
Let $N$ be a von Neumann algebra. An element $t\in N$ is of \emph{class 1} if $t=1$ and of \emph{class 2} if $t\neq 1$ and $(t-1)^2=0$. 
\end{df}

We start with a few standard facts on von Neumann factors. Most of the properties are trivial and are included merely as a reference.

\begin{lem}\label{lem.vonN}
Let $N$ be a von Neumann factor. Then the following holds:
\begin{enumerate}
\item\label{vonN1} For each $x\in N$, $\frac{x}{2}, \frac{x}{3}\in N$ ($\frac{x}{2}$ is the unique $y\in N$ satisfying $2y=x$).
\item\label{vonN2} For each $0\neq n\in N$ such that $n^2=0$ there exists $e,f\in I(N)$ such that $n=enf$ and $fe=0$. 
\item\label{vonN3} For each $0\neq n\in N$, $n^2=0$ there exists idempotents $e,g\in I(N)$ and $k\in N$ such that $n=eng$, $ge=eg=0$, $e=nk$ and $kn=g$.
\item\label{vonN4} For each $u=2e-1\in \Inv(N)$ and $r\in C^2(u)$ there exists elements $z_1,z_2\in Z(N)$ such that $r=z_1e+z_2(1-e)$.
\item\label{vonN5} For each $e\in I(N)$ and $d_1\in eNe$ we have that if $d_1x=xd_1$ for all $x\in GL(eNe)$, then $d_1=ze$ for some $z\in Z(N)$.
\item\label{vonN6} For $e,g\in I(N)$ from \eqref{vonN3}, $f\mydef 1-e-g$ and $a_3\in eNf$, $a_4\in fNf$, $a_5\in fNg$ we have $a_3(fNg)=0 \Rightarrow a_3=0$, $a_4(fNg)=0 \Rightarrow a_4=0$ and $(eNf)a_5=0 \Rightarrow a_5=0$.
\item\label{vonN7} If $t\in N$, $C(t)=GL(N)$ and $(t-2)(t+1)=0$ then $t=2$ or $t=-1$.
\end{enumerate}
\end{lem}
\begin{proof}
Left to the reader.
\end{proof}

We now present a few lemmata  constituting the proof of Theorem~\ref{thmC}. Notice, each class 2 element $s$ is invertible with inverse $2-s$, see Definition~\ref{def.classtwo}.

\begin{lem}\label{lem1.2}
Let $N$ be a von Neumann factor. Then for each $s\in N$ of class 2 there exists $u\in Inv(N)$ and $r\in C^2(u)$ such that $usu=s^{-1}$, $rsr^{-1}=s^2$ and $s^3\neq 1$.
\end{lem}
\begin{proof}
Fix $s=n+1$ and select $e,f$ as in Lemma~\ref{lem.vonN}\eqref{vonN2}. Since $fe=0$ we get $f=f(1-e)$, so $n=enf=en(1-e)$. It follows that $u\mydef e-(1-e)$ satisfies $u\in Inv(N)$ and $usu=1+unu=1+uen(1-e)u=1-en(1-e)=s^{-1}$. Set $r\mydef1+e\in GL(N)$ with $r^{-1}=1-\frac{e}{2}$. Now, if $t\in C(u)$ then $tu=ut$, $te=et$ and $tr=rt$ so $r\in C^2(u)$. Moreover, $rsr^{-1}=1+ren(1-e)r^{-1}=1+2en(1-e)=1+2n=s^2$. Since $s^3=1+3n$, and $n\neq 0$ we get $s^3\neq 1$.
\end{proof}

\begin{remark}
Recall two idempotents $e$ and $g$ are \emph{orthogonal} if $eg=ge=0$. For two such idempotents $f\mydef1-e-g$ is also an idempotent and one can write each element $t\in N$ as the sum $t=ete+etg+etf+get+gtg+gtf+fte+ftg+ftf$. Notice that $ete\in eNe$ and so on. This may be recorded as
\begin{align*}
t=\tmatrix{ete}{etg}{etf}{gte}{gtg}{gtf}{fte}{ftg}{ftf},
\end{align*}
with respect to the ordering $e,g,f$. Since the idempotents are pairwise orthogonal one can perform classical matrix multiplication to get products.
\end{remark}

\begin{lem}\label{lem1.4}
Let $N$ be a von Neumann factor. Then for each $s=1+n$ of class 2 and $t\in GL(N)$ we have that $t\in C(s)$ iff there exists pairwise orthogonal idempotents $e,g,f\in I(N)$ and elements $k, t_1,\dots, t_5 \in N$ such that $e+g+f=1$, $e=nk$, $g=kn$, $n=eng$ and
\begin{align}\label{matrixthree}
t=\tmatrix{t_1}{t_2}{t_3}{0}{kt_1n}{0}{0}{t_5}{t_4},
\end{align}
with respect to the order $e,g,f$ (so $t_2\in eNg$ and so on). \end{lem}
\begin{proof}
Fix $s=1+n$ of class 2 and $t\in GL(N)$. By assumption on $n$ there exist $e,g,k$ as in Lemma~\ref{lem.vonN}\eqref{vonN3}. Define $f\mydef1-e-g$. 

Suppose $t\in C(s)$. Then $ts=st$ and $tn=nt$. Using $gn=gen=0$ we get $gte=gtnk=gntk=0$. Similarly $fte=ftnk=fntk=0$ and $gtf=kntf=ktnf=0$. Finally, using $en=n=ng$ we get the desired form of $t$ via $k(ete)n=ket(en)g=ketng=kentg=kntg=gtg$.

Conversely, for any $t\in GL(N)$ of the form as in \eqref{matrixthree} with $e=nk$, $g=kn$ and $n=eng$. We get $st=ts$ from
\begin{align*}
n=\tmatrix{0}{eng}{0}{0}{0}{0}{0}{0}{0}, \ \ nt=\tmatrix{0}{nkt_1n}{0}{0}{0}{0}{0}{0}{0}, \ \ tn=\tmatrix{0}{t_1n}{0}{0}{0}{0}{0}{0}{0},  \ \ nkt_1n=et_1n=t_1n.\ \ 
\end{align*}
\end{proof}

\begin{lem}\label{lem1.5}
Let $N$ be a von Neumann factor. Then for each $s=1+n$ of class 2 and $d\in C^2(s)$ there exists elements $z_1, z_2\in Z(N)$ such that $$d=z_11+z_2n.$$
\end{lem}
\begin{proof}
Fix $s=1+n$ of class 2 and $d\in C^2(s)$. Since $d\in C^2(s)\subseteq C(s)$, there exists $e,g,f\in I(N)$ and elements $k, d_1\,\dots,d_5\in N$ representing $d$ as in \eqref{matrixthree}. Using Lemma~\ref{lem1.4} on the following $t_1, t_2, t_3\in GL(N)$,
\begin{align*}
t_1\mydef\tmatrix{x}{0}{0}{0}{kxn}{0}{0}{0}{f}, \ t_2\mydef\tmatrix{e}{0}{0}{0}{g}{0}{0}{y}{f}, \ t_3\mydef\tmatrix{e}{0}{z}{0}{g}{0}{0}{0}{f}, \ d=\tmatrix{d_1}{d_2}{d_3}{0}{kd_1n}{0}{0}{d_5}{d_4},
\end{align*}
with (any) $x\in GL(eNe)$, $y\in fNg$, $z\in eNf$ we have that each $t_i\in C(s)$.  Since $d\in C^2(s)$, $d$ commutes with each $t_i\in C(s)$. In particular, $dt_1=t_1d$, so $d_1x=xd_1$ in $eNe$ for all $x\in GL(eNe)$. Using Lemma~\ref{lem.vonN}\eqref{vonN5} we know $d_1=z_1e$ for some $z_1\in Z(N)$. Now, since $kd_1n=k(z_1e)n=z_1kn=z_1g$,
\begin{align*}
d=\tmatrix{d_1}{d_2}{d_3}{0}{kd_1n}{0}{0}{d_5}{d_4}=\tmatrix{z_1e}{d_2}{d_3}{0}{z_1g}{0}{0}{d_5}{d_4}=z_11+d'  \ \ \text{ for } \ \ d'\mydef\tmatrix{0}{d_2}{d_3}{0}{0}{0}{0}{d_5}{d_4-z_1f}.
\end{align*}
Successively using $d't_2=t_2d'$ and $d't_3=t_3d'$ we get $d_3=0$, $d_4-z_1f=0$ and $d_5=0$ (by Lemma~\ref{lem.vonN}\eqref{vonN6}).  For $r\mydef n+k+f$ the equality $d't_1r=t_1d'r$ reduces to $x(d_2k)=(d_2k)x$ in $eNe$. By Lemma~\ref{lem.vonN}\eqref{vonN5}, $d_2k=z_2e$ for some $z_2\in Z(N)$. Consequently $d'=d_2=d_2g=d_2kn=z_2en=z_2n$, so $d=z_11+z_2n$.
\end{proof}

\begin{prop}\label{lem1.6}
Let $N$ be a von Neumann factor. Then for each $s\in N$, the element $s$ is of class 2 if and only if \eqref{class.1}-\eqref{class.4} is satisfied:
\begin{enumerate}
\item\label{class.1} If $t\in GL(N)$, then $C(s)\subsetneq C(t) \Leftrightarrow GL(N)=C(t)$.
\item\label{class.2} There exists $u\in Inv(N)$ such that $usu=s^{-1}$, and
\item\label{class.3} an elements $r\in C^2(u)$ such that $rsr^{-1}=s^2$.
\item\label{class.4} $s^3\neq 1$.
\end{enumerate}
\end{prop}
\begin{proof}
Fix $s=1+n\in GL(N)$. Suppose $s$ is of class 2. Then \eqref{class.2}-\eqref{class.4} holds by Lemma~\ref{lem1.2}. To show \eqref{class.1} fix any $t\in GL(N)$.

``$\Leftarrow$'': Since $C(t)=GL(N)$ it suffices to show $C(s)\neq C(t)$. Assuming $C(s)= C(t)$ we have $GL(N)=C(t)=C(s)=C(n)$, but $r\not\in C(n)$ because $rnr^{-1}=s^2-1=(1+n)^2-1=2n\neq n$.

``$\Rightarrow$'' Since $C(s)\subseteq C(t)$, we get $t\in C^2(s)= \{t' : xt'=t'x \ \text{for all}\ x\in C(s)\}$. By Lemma~\ref{lem1.5}, $t=z_11+z_2n$ for some $z_1, z_2\in Z(N)$. Assume $z_2n\neq 0$. Then $x\in GL(N)$ commutes with $t$ iff it commutes with $s=1+n$. This implies $C(s)=C(t)$, giving a contradiction. We deduce that $z_2n=0$, so $t=z_11\in Z(N)$ and $GL(N)=C(t)$.

Suppose $s=1+n\in GL(N)$ satisfy \eqref{class.1}-\eqref{class.4}. Set $t=s+s^{-1}$. By \eqref{class.2} we get $u\in Inv(N)$ such that $utu=t$. By \eqref{class.3}, $r\in C^2(u)$. Hence property Lemma~\ref{lem.vonN}\eqref{vonN4} provides elements $z_1,z_2\in Z(N)$ such that $r=z_1e+z_2(1-e)$ for $e\mydef\frac{u+1}{2}$. We know that $t$ commutes with $u$, $e$, and $r$. It follows (by \eqref{class.3}) that $t=rtr^{-1}=r(s+s^{-1})r^{-1}=s^2+s^{-2}=(s+s^{-1})^2-2=t^2-2$, so $t(t-1)=2$ implying $t\in GL(N)$.

We now show $C(s)\subsetneq C(t)$: ``$C(s)\subseteq C(t)$'': If $x\in C(s)$ then $s^{-1}(xs)s^{-1}=s^{-1}(sx)s^{-1}$ so $x\in C(s^{-1})$ and $x\in C(t)$.
``$C(s)\neq C(t)$'': We know $u\in C(t)$. Assuming $u\in C(s)$, we get $s=s^{-1}$ by \eqref{class.2}. Using \eqref{class.3}, $1=s^2=rsr^{-1}$, so $s=1$. Hence $GL(N)=C(s)$ and by \eqref{class.1}, $C(s)\subsetneq C(s)$ (false). So $u\not\in C(s)$. 

Knowing $C(s)\subsetneq C(t)$ we get $GL(N)=C(t)$ by \eqref{class.1}. Since we already established $(t-2)(t+1)=0$, we get $t=2$ or $t=-1$ by Lemma~\ref{lem.vonN}\eqref{vonN7}. If $t=-1$ then $-1=s+s^{-1}$, $s^3=(1+s^{-1})s(1+s^{-1})=(s+1)(1+s^{-1})=2+t=1$, contradicting \eqref{class.4}. Therefore $t=2$. Hence $s\neq 1$ and $(s-1)^2=s^2+1-2s=s(s+s^{-1}-2)=s(t-2)=0$. So $s$ is of class 2.
\end{proof}

\smallskip
\emph{Proof of Theorem~\ref{thmC}:} Since $1\neq s\in GL(N)$, $s$ is of class 2 iff $(s-1)^2=0$. The result now follows from Proposition~\ref{lem1.6}. \qed

\begin{remark}
Theorem~\ref{thmC} remains valid in greater generality. In fact, for any unital ring $N$ satisfying properties \eqref{vonN1}-\eqref{vonN7} of Lemma~\ref{lem.vonN}, we have the same characterisation of elements of class 2 in terms of the general linear group. This is because no other properties of $N$ as a unital ring were used in the proof. 
\end{remark}

\section{$\Delta$-sets and the proof of Theorem~\ref{thmB}}

In this section we prove the main result. This is a characterisation of a natural equivalence relation on the set of idempotents, but expressed via properties of the general linear group.

Let $N$ be a von Neumann algebra. We let  $P(N)$ denote the set of \emph{projection} elements in $N$, i.e.,
\begin{align*}
&P(N)\mydef\{p\in N : p=p^*=p^2\}.
\end{align*}

\begin{df}[cf.~\cite{MR0052795, MR0081885}]\label{def2.1}
Let $N$ be a von Neumann algebra containing an idempotent $e$, an involution $u$ and a nonempty subset $\phi$ of involutions. Set $\phi^2\mydef \{uv : u,v\in \phi\}$ and
\begin{align*}
&I^{\pm}(u)\mydef \{x\in N : ux=\pm x\}, \ I^+(\phi^2)\mydef\{x\in N: tx=x \ \text{for all}\ t\in \phi^2\},\\
&N(\phi)\mydef\{v\in \Inv(N): v\phi = \phi v\}, \ \ \ \Delta^\pm(e)\mydef \{v\in \Inv(N) : I^{\pm}(v)=eN\}.
\end{align*}
\end{df}
\begin{df}[cf.~\cite{MR0052795, MR0081885}]\label{delta.set}
Let $N$ be a von Neumann algebra containing a nonempty subset $\phi$ of involutions. We say $\phi$ is a $\Delta^+$-set (resp. $\Delta^-$-set) if $\phi=\Delta^+(e)$ (resp. $\Delta^-(e)$) for some idempotent $e$. We call $\phi$ a $\Delta$-set if it is a $\Delta^+$-set or a $\Delta^-$-set for some $e\in I(N)$. 
\end{df}

As in the previous section we include a few standard facts on von Neumann factors. Two idempotents $e,f$ are \emph{similar} if $e=ufu^{-1}$ for some invertible element $u$.

\begin{lem}\label{lem.vonN.2}
Let $N$ be a von Neumann factor. Then the following holds:
\begin{enumerate}
\item\label{vonN0} For each $x\in N$, $\frac{x}{2}\in N$ ($\frac{x}{2}$ is the unique $y\in N$ satisfying $y+y=x$).
\item\label{vonN.1} $P(N)$ is partially ordered via $p\leq q$ iff $pq=qp=p$.
\item\label{vonN.2} If $x\in N$ then $p_x\mydef \inf\{p\in P(N) : x=px\} \in P(N)$ and $x=p_xx$.
\item\label{vonN.3} If $y\in N$ then $q_y\mydef \inf\{q\in P(N) : y=yq\} \in P(N)$ and $y=yq_y$.
\item\label{vonN.4} If  $X\subseteq N$ then $q_X\mydef \sup \{q_x : x\in X\} \in P(N)$.
\item\label{vonN.5} If $yx=0$ for $y,x\in N$ then $q_yp_x=0$.
\item\label{vonN.6} If $y,x\in N$ are both nonzero then $yNx\neq \{0\}$.
\item\label{vonN.7} If $x\in N$ then $p_x=uu^*$ and $q_x=u^*u$ for some $u=uu^*u\in N$.
\item\label{vonN.8} Every idempotent in $N$ is similar to a projection in $N$.
\end{enumerate}
\end{lem}
\begin{proof}
Left to the reader.
\end{proof}

We recall a few standard facts. These are purely algebraic observations. The proofs are included for completeness.

\begin{lem}[cf.~\cite{MR0052795, MR0081885, MR0120174}]\label{lem2.4}
Let $N$ be a von Neumann algebra containing an idempotent $e$. Then
\begin{enumerate}
\item\label{max.0} If $u\mydef2e-1$ then $u\in \Inv(N)$, $I^{+}(u)=eN$ and $I^{-}(u)=(1-e)N$.
\item\label{max.1} If $u\in \Inv(N)$ then $I^{+}(u)+I^{-}(u)=N$ and $I^{+}(u)\cap I^{-}(u)=\{0\}$.
\item\label{max.2} If $u,v\in \Inv(N)$ satisfy $I^{\pm}(u)=I^{\pm}(v)$ then $u=v$.
\item\label{max.3} If $e,f\in I(N)$ then $eN=fN \Leftrightarrow f=e+ey(1-e)$ for some $y\in N$.
\end{enumerate}
\end{lem}
\begin{proof}
\eqref{max.0}: If $e\in I(N)$ then $(2e-1)^2=1$, so $2e-1\in \Inv(N)$. Set $u\mydef 2e-1$. Then 
$I^+(u)=\{x \in N : (2e-1)x=x \}=\{x \in N : 2ex=2x \}=eN$ and $I^-(u)=\{x \in N : 2ex=0 \}=\{x \in N : (1-e)x=x \}=(1-e)N$.

\eqref{max.1}: If $x\in N$, set $x_\pm\mydef\frac{x\pm ux}{2}\in N$. Then $ux_\pm=\pm x_\pm$, so $x=x_++x_-$, $x_+\in I^+(u)$ and $x_-\in I^-(u)$ giving $N=I^+(u)+I^-(u)$. If $x\in I^+(u)\cap I^-(u)$ then $ux=x=-(-x)=-ux$ so $2ux=0$. Using $u^2=1$ we get $x=0$.

\eqref{max.2}: For suitable $e,f$ such that $u=2e-1$ and $v=2f-1$ we get $eN=fN$ and $(1-e)N=(1-f)N$. Hence $ef=f$ and $(1-e)(1-f)=1-f$. The latter gives $1-e-f+ef=1-f$. So $ef=e$, $f=ef=e$ and $u=v$.

\eqref{max.3}:  `$\Leftarrow$' $fN=(e+ey(1-e))N\subseteq eN$, and $fe=e$ so $eN=feN\subseteq fN$. `$\Rightarrow$' Using $e=ee=fx$ for some $x$, $fe=f(fx)=fx=e$. Similarly $ef=f$. One now verify $f=e+ey(1-e)$ for $y\mydef f-e$.
\end{proof}

We are now in position to establish properties of a $\Delta^+$-set which (up to maximality) characterise the $\Delta^+$-set.

\begin{lem}\label{lem2.5}
Let $N$ be a von Neumann factor. Then every $\Delta^+$-set $\phi$ satisfies the following properties: 
\begin{enumerate}
\item\label{delta.aa} If $u,v,w\in \phi$ then $uvw=wuv\in \phi$.
\item\label{delta.bb} If $u,v\in \phi$ then there exist a unique $w\in \phi$ such that $wvw=u$.
\item\label{delta.cc} If $u\in \Inv(N)$, then $u\in N(\phi)$ iff $wuw=u$ for some $w\in\phi$.
\item\label{delta.dd} If $t\in N$ is an elements of $\phi^2$ then $t$ is of class 1 or class 2.
\end{enumerate}
\end{lem}
\begin{proof}
\eqref{delta.dd}: Fix $t\in \phi^2$. Select $u,v\in \phi$ such that $t=uv$. Since elements of a $\Delta^+$-set have the same $I^+$-set,  $I^{+}(u)=I^{+}(v)$. For $u=2e-1$ and $v=2f-1$, $eN=I^{+}(u)=I^{+}(v)=fN$ by Lemma~\ref{lem2.4}\eqref{max.0}. If follows that $ef=f$ and $fe=e$, see Lemma~\ref{lem2.4}\eqref{max.3}. Hence $uv=(2e-1)(2f-1)=4ef-2f-2e+1=2(f-e)+1$ and $(t-1)^2=4(f-e)^2=0$. So $t$ is of class 1 or class 2.

\eqref{delta.aa}: Using $u,v,w\in \phi$, $I^{+}(u)=I^{+}(v)=I^{+}(w)$. Similarly to the proof of \eqref{delta.d}, $uvw= (2e-1)(2f-1)(2g-1)= 2(e-f+g)-1= u-v+w$ for suitable $e,f,g\in I(N)$ (using $eN=fN\Rightarrow ef=f$, etc.). This implies $uvw=wvu$, so $uvw\in \Inv(N)$. Now $I^{+}(uvw)=(e-f+g)N=(e-ef+eg)N\subseteq eN=I^{+}(u)$ by Lemma~\ref{lem2.4}, and conversely if $x\in I^{+}(u)=I^{+}(v)=I^{+}(w)$, then $uvwx=x$ and $x\in I^{+}(uvw)$. Since $u\in\phi$ and $I^{+}(u)=I^{+}(uvw)$, also $uvw\in\phi$.

\eqref{delta.bb}: Fix any $u,v\in \phi$. As in the proof of \eqref{delta.d},  $uv=2(f-e)+1$. By symmetry $vu=2(e-f)+1$, so $uv+vu=2$ and $uvu+v=2u$ (using $u^2=1$). For $w\mydef \frac{u+v}{2}$ we get $4wvw=(uv+1)(u+v)=uvu+2u+v=4u$ (using $uvu+v=2u$), so $wvw=u$. Similarly $4w^2=(u+v)(u+v)=4$ (using $uv+vu=2$), so $w\in Inv(N)$. Since $2w=u+v=2(e+f)-2$ it follows that $I^{+}(w)=(e+f)N=(e+ef)N\subseteq eN=I^{+}(u)$. Conversely if $x\in I^{+}(u)=I^{+}(v)$, then $(u+v)x=2x$ and $x\in I^{+}(w)$. Therefore $w\in \phi$. For the uniqueness take $u,w,v\in \phi$ such that $wvw=u$. As in \eqref{delta.a} we get $wvw=w-v+w$. Hence $u=w-v+w$ and $w=\frac{u+v}{2}$.

\eqref{delta.cc}: Fix $u\in \Inv(N)$. If $u\in N(\phi)$, then $uw=w^\star u$ for some $w,w^\star\in \phi$, so $uwu=w^\star$. Using $\eqref{delta.b}$, $w^{\dag}ww^{\dag}=w^\star$ for some $w^{\dag}\in \phi$. As in $\eqref{delta.a}$ we get $w^{\dag}ww^{\dag}=2w^{\dag}-w$. Hence $uwu=w^\star=w^{\dag}ww^{\dag}=2w^{\dag}-w$, so $w^{\dag}=\frac{uwu+w}{2}$ commutes with $u$. Conversely, suppose $vu=uv$ for some $v\in\phi$. To show $u\phi=\phi u$, fix any $w\in \phi$. Since $v$ and $w$ belong to the same $\Delta^+$-set $I^{+}(v)=I^{+}(w)$. For $w=2e-1$, $v=2f-1$, $eN=I^{+}(w)=I^{+}(v)=fN$ by Lemma~\ref{lem2.4}. We have $w^\star\mydef2(ueu)-1\in \Inv(N)$ using $ueu\in I(N)$. Since $uv=vu$ we get $uf=fu$ and $I^{+}(w^\star)=(ueu)N=ueN=ufN=fN$ by Lemma~\ref{lem2.4}. Now $v\in\phi$ and $I^{+}(w^\star)=I^{+}(v)$ giving that $w^\star\in\phi$. Using $uwu=w^\star$ we get $uw=w^\star u$. It follows that $u\phi=\phi u$, i.e., $u\in N(\phi)$.
\end{proof}

The following lemmata relies heavily on the properties listed in Lemma~\ref{lem.vonN.2}. In \cite{MR0081885} similar results were established using dimension theory, irreducibility, regularity and lattice properties of continuous rings, see \cite[Proposition~3 and Proposition~8]{MR0081885}.

\begin{lem}\label{lem2.6}
Let $N$ be a von Neumann factor. Suppose $e\in I(N)$, $B\subseteq Ne$ and $A\mydef \{x\in eN : bx=0 \ \text{for all} \ b\in B\}$ satisfy $A\neq \{0\}$ and $vA=A$ for all $v\in \Inv(eNe)$. If $e=e^*$, then $A=eN$.
\end{lem}
\begin{proof}
Using Lemma~\ref{lem.vonN.2}\eqref{vonN.4} set $q\mydef q_B \in N$. The proof consist of five steps.

Step 1.~``{\color{black}$A=\{x\in N : x=ex=(1-q)x\}$}'': Fix $x\in N$ such that $x=ex=(1-q)x$. Using Lemma~\ref{lem.vonN.2}\eqref{vonN.1}-\eqref{vonN.2}, $p_x\leq 1-q$. Using Lemma~\ref{lem.vonN.2}\eqref{vonN.3}, $b=bq_b$. Using Lemma~\ref{lem.vonN.2}\eqref{vonN.2}, $x=p_xx$. Using Lemma~\ref{lem.vonN.2}\eqref{vonN.4}, $q_b\leq q$ for $b\in B$. So $bx=bq_bq(1-q)p_xx=0$. Hence $x\in eN$ and $bx=0$ for all $b\in B$, i.e., $x\in A$. Conversely, fix $x\in eN$ such that $bx=0$ for all $b\in B$. By Lemma~\ref{lem.vonN.2}\eqref{vonN.5}, $q_bp_x=0$ for $b\in B$, so $p_x\leq 1-q_b$. Then $q_b\leq 1-p_x$ for $b\in B$. Using Lemma~\ref{lem.vonN.2}\eqref{vonN.4}, $q\leq 1-p_x$, so $(1-q)p_x=p_x$. It follows that $x=p_xx=(1-q)p_xx=(1-q)x$ and $x=xe$.

Assume $q\neq 0$. Using that $A\neq \{0\}$ select a nonzero $z\in A$. Using Lemma~\ref{lem.vonN.2}\eqref{vonN.3}, $q_{z^*}\neq 0$. Using Lemma~\ref{lem.vonN.2}\eqref{vonN.6}, there exists a nonzero element $y\in q_{z^*}Nq$. 

Step 2.~``{\color{black}$y\in eNe$}'': For each $b\in B\subseteq Ne$ we have $b=be$. We now have that $b(1-e)=0$, so $q_bp_{1-e}=0$ by Lemma~\ref{lem.vonN.2}\eqref{vonN.5}. It follows that $q_b\leq 1-p_{1-e}$, $q\leq 1-p_{1-e}$, $qp_{1-e}=0$, $qp_{1-e}(1-e)=0$, $q(1-e)=0$, so $q=qe$. Consequently, $y\in q_{z^*}Nq=Nqe\subseteq Ne$. Since $z\in A\subseteq eN$ and $e=e^*$ we have $z^*=(ez)^*=z^*e$, so $z^*(1-e)=0$. Using Lemma~\ref{lem.vonN.2}\eqref{vonN.5}, $q_{z^*}(1-e)=0$, so $(1-e)q_{z^*}=0$ and $eq_{z^*}=q_{z^*}$. So $y\in q_{z^*}Nq\subseteq eN$. We conclude $y\in eNe$.

Step 3.~``{\color{black}$p_y \ \bot \ q_y$}'': Using Lemma~\ref{lem.vonN.2}\eqref{vonN.2} on $y=q_{z^*}y$ ($y\in q_{z^*}Nq$), $p_y\leq q_{z^*}$. Using Lemma~\ref{lem.vonN.2}\eqref{vonN.3} on $y=yq$ ($y\in q_{z^*}Nq$), $q_y\leq q$. Since $z=(1-q)z$ we have $qz=0$ and $z^*q=0$. Using Lemma~\ref{lem.vonN.2}\eqref{vonN.5}, $q_{z^*}q=0$, so $q_{z^*}\leq 1-q$. Hence $p_y \leq q_{z^*} \leq 1-q$ and $q_y\leq q$. We conclude $p_y \ \bot \ q_y$ (i.e., $p_yq_y=0$).

By Lemma~\ref{lem.vonN.2}\eqref{vonN.7}, there exists $u=uu^*u\in N$ such that $p_y=uu^*$ and $q_y=u^*u$. Set 
$$v\mydef e(u+u^*+1-p_y-q_y)e.$$ 

Step 4.~``{\color{black}$v^2=e$}'': First we show $u, u^*, p_y,q_y\in eNe$. Using $u=uu^*u$, $p_y=uu^*$, $p_y\leq q_{z^*}$ and $eq_{z^*}=q_{z^*}$ we have $u=eq_{z^*}p_yu=eu$. Using $u=uu^*u$, $q_y=u^*u$, $q_y\leq q$ and $qe=q$ we have $u=uq_yqe=ue$. So $eue=u$. Using $e=e^*$, $eu^*e=u^*$. We have $eq_ye=q_y$ using that $q_y\leq q\leq e$ (from $q=qe$ and $y=yq$). We have $ep_ye=p_y$ using that $p_y\leq q_{z^*}\leq e$ (from $y=q_{z^*}y$ and $eq_{z^*}=q_{z^*}$). So $v=u+u^*+e-p_y-q_y$.
It is clear that $(u+u^*)$ is orthogonal to $e-(p_y+q_y)$ because $(u+u^*)e=(u+u^*)$ and $(u+u^*)(p_y+q_y)=uq_y+u^*p_y=(u+u^*)$ using that $p_y \ \bot \ q_y$ and $u=p_yuq_y$. Similary we deduce $uu=uq_yp_yu=0$, so $u^2=0$, $(u^*)^2=0$. We conclude 
$$v^2=(u+u^*)^2+(e-(p_y+q_y))^2=uu^*+u^*u+e-(p_y+q_y)=e.$$

Step 5.~``{\color{black}$y=0$}'': By assumption applied to $v\in Inv(eNe)$ we get $vA=A$. Hence $vz\in A$ so $vz=(1-q)vz$ giving $qvz=0$. Hence $uqvz=0$. Using $u=uq_y$ and $q_y\leq q$, $uq=(uq_y)q=u(q_yq)=uq_y=u$, so $uvz=0$. Using $u=uq_y=ue$ and $v=p_yu+q_yu^*+e-p_y-q_y$ we get $uv=0+uu^*+u-0-u=p_y$. So $p_yz=0$ and hence $z^*p_y=0$. Using Lemma~\ref{lem.vonN.2}\eqref{vonN.5}, $q_{z^*}p_y=0$. But $p_y\leq q_{z^*}$, so $p_y=0$. We conclude $y=p_yy=0y=0$.

But this contradicts the fact that $y\neq 0$. We conclude $q=0$, $A=eN$.
\end{proof}

\begin{lem}\label{lem2.7}
Let $N$ be a von Neumann factor. Suppose $e\in I(N)$, $B\subseteq Ne$ and $A\mydef \{x\in eN : bx=0 \ \text{for all} \ b\in B\}$ satisfy $A\neq \{0\}$ and $vA=A$ for all $v\in \Inv(eNe)$. Then $A=eN$.
\end{lem}
\begin{proof}
By Lemma~\ref{lem.vonN.2}\eqref{vonN.8} any idempotent is similar to a projection. Hence we can find $u\in GL(N)$ such that $p=ueu^{-1}\in P(N)$. Set 
\begin{align}\label{text.0}
B'\mydef uBu^{-1}.
\end{align}
Using $B\subseteq Be$ we get $B'=uBu^{-1}\subseteq uBeu^{-1}=uBu^{-1}ueu^{-1}=B'p$. Set
\begin{align}\label{text.0.1}
A'\mydef uAu^{-1}.
\end{align}
Since $A\neq \{0\}$ we get $A'\neq \{0\}$. Since $u(eN)u^{-1}=ueN=puN=pN$ we get
\begin{align*}
A'&= \{uxu^{-1}\in ueN u^{-1} : bx=0 \ \text{for all} \ b\in B\}\\
&= \{uxu^{-1}\in pN : ubu^{-1}uxu^{-1}=0 \ \text{for all} \ b\in B\}\\
&= \{y\in pN : ubu^{-1}y=0 \ \text{for all} \ b\in B\}\\
&= \{y\in pN : ubu^{-1}y=0 \ \text{for all} \ ubu^{-1}\in uBu^{-1}\}\\
&= \{y\in pN : b'y=0 \ \text{for all} \ b'\in B'\}
\end{align*}
We claim that $wA'=A'$ for all $w\in \Inv(pNp)$. To see this fix $w\in \Inv(pNp)$. Notice that $\Inv(pNp)=u\Inv(eNe)u^{-1}$ because if $v\in \Inv(eNe)$, then 
\begin{align*}
uvu^{-1}=ueveu^{-1}=(ueu^{-1})uvu^{-1}(ueu^{-1})\in \Inv(pNp),
\end{align*}
and similarly for the converse containment. Hence $w=uvu^{-1}$ for some $v\in \Inv(eNe)$. By assumption we know that $vA=A$. It follows that $wA'=uvu^{-1}uAu^{-1}=A'$.

Having that $p\in I(N)$, $B'\subseteq Np$ and $A'= \{y\in pN : b'y=0 \ \text{for all} \ b'\in B'\}$ satisfy $wA'=A'\neq \{0\}$ for all $w\in \Inv(pNp)$, and $p=p^*$, we can now apply Lemma~\ref{lem2.6} to deduce $A'=pN$. It now follows that 
\begin{align*}
A&=u^{-1}uAu^{-1}u=u^{-1}pNu=u^{-1}(ueu^{-1})N=eN. \qedhere
\end{align*}
\end{proof}

\begin{lem}\label{lem2.8}
Let $N$ be a von Neumann factor. Suppose $u\in \Inv(N)$, $B\subseteq N$ and $A\mydef \{x\in N : bx=0 \ \text{for all} \ b\in B\}$ satisfy $vA=A$ for all $v\in C(u)\cap \Inv(N)$. Then $A\in\{\{0\}, I^{+}(u), I^{-}(u), N\}$.
\end{lem}
\begin{proof}
Define $A^\pm\mydef A\cap I^{\pm}(u)$ and $e\mydef \frac{1+u}{2}$. By Lemma~\ref{lem2.4}, $e\in I(N)$ and $eN=I^{+}(u)$.

Fix any $v\in \Inv(eNe)$. Define $w\mydef v+(1-e)$. Since $v^2=e$ and $v=eve$ we get $w^2=v^2+(1-e)^2=1$, so $w\in \Inv(N)$. Also, using $ew=we$, we get $uw=wu$ (recall $u=2e-1$), so $w\in C(u)\cap \Inv(N)$. Hence, by assumption $wA=A$. Notice, if $x\in eN$ then $x=ex$, so $w(eN\cap A)=we(eN\cap A)$. Now using $ew=v=we$, we get 
\begin{align*}
vA^+&=v(eN\cap A)=we(eN\cap A)=w(eN\cap A)=weN\cap wA\\
&=ewN\cap A=eN\cap A=A^+.
\end{align*}

Recalling $A=\{x\in N : bx=0 \ \text{for all} \ b\in B\}$ and $A^+=eN\cap A$ it follows that $A^+=\{x\in eN : bx=0 \ \text{for all} \ b\in B\}$. Moreover, we can replace $B$ by $Be$ without enlarging $A^+$: indeed if $x\in eN$ satisfies $bx=0$ for all $b\in Be$, then every $b'\in B$ satisfies $b'x=b'(ex)=(b'e)x=0$. Consequently,
$$A^+= \{x\in eN : bx=0 \ \text{for all} \ b\in Be\}.$$
Assuming $A^+\neq \{0\}$ (and using that $vA^+=A^+$ for each $v\in \Inv(eNe)$) we can now apply {\color{black}Lemma~\ref{lem2.7}} to $A^+$, $Be$ in place of $A,B$ to deduce $A^+=eN$. Including the trivial case we have $A^+\in\{\{0\}, I^+(u)\}$. 

In a similar fashion (using $I^-(u)=(1-e)N$, see Lemma~\ref{lem2.4}) we can prove that $vA^-=A^-$ for each $v\in \Inv((1-e)N(1-e))$ and use {\color{black}Lemma~\ref{lem2.7}} on $A^-$, $B(1-e)$ to deduce $A^-\in\{\{0\}, I^-(u)\}$. 

By Lemma~\ref{lem2.4}, $I^{+}(u)+I^{-}(u)=N$ and $I^{+}(u)\cap I^{-}(u)=\{0\}$. Hence $A^++A^-=A$ and $A^+\cap A^-=\{0\}$. If follows that $A\in\{\{0\}, I^{+}(u), I^{-}(u), N\}$.
\end{proof}

\begin{lem}\label{lem2.9}
Let $N$ be a von Neumann factor containing a subset $\phi$ of involutions satisfying \eqref{delta.aa}-\eqref{delta.dd} of Lemma~\ref{lem2.5} and $|\phi|>1$. Define $A\mydef I^+(\phi^2)$. Then $A=eN$ for some $e\in I(N)$ and $\phi\subseteq \Delta^+(e)$ or $\phi\subseteq \Delta^-(e)$.
\end{lem}
\begin{proof}
The proof consist of five steps.

Step 1.~``$A\neq N$'': If $A=N$ then $1\in N=\{x\in N: tx=x \ \text{for all}\ t\in \phi^2\}$, so $t1=1$ for each $t\in \phi^2$. But $\phi^2=\{1\}$ implies $|\phi|\leq 1$ (if $uv=1\Rightarrow u=v$).

Step 2.~``$A\neq \{0\}$'': Clearly $1\in \phi^2$ using $u\in \phi \Rightarrow 1=u^2\in \phi^2$. Also $\phi^2\neq \{1\}$, see Step 1. One can therefore select $t_0\in \phi^2$ such that $t_0\neq 1$. By property Lemma~\ref{lem2.5}\eqref{delta.dd}, $t_0=1+n_0$ is of class 2, so $n_0\neq 0$ and $n_0^2=0$. Select any $t=1+n\in \phi^2$. Then $t=uv$, $t_0=u_0v_0$ for some $u,v,u_0,v_0\in \phi$. By property Lemma~\ref{lem2.5}\eqref{delta.aa}, $tt_0=uvu_0v_0=uv_0u_0v=u_0v_0uv=t_0t\in \phi^2$. Now property Lemma~\ref{lem2.5}\eqref{delta.dd} ensures $tt_0$ is of class 1 or 2, so $(tt_0-1)^2=0$. Using $n,n_0$ commute (recall $tt_0=t_0t$, $t=1+n$, $t_0=1+n_0$) gives
$$0=(tt_0-1)^2=((n+1)(n_0+1)-1)^2=(n+n_0+nn_0)^2=2nn_0.$$
We deduce that $tn_0=(1+n)n_0=n_0$ for all $t\in \phi^2$. Consequently, $n_0$ is an element of $\{x\in N: tx=x \ \text{for all}\ t\in \phi^2\}=A$.

Step 3.~``$sA=A$ for all $s\in N(\phi)$": Using $s\phi=\phi s$ we get 
\begin{align*}
sA&= \{sx\in N: tx=x \ \text{for all}\ t\in \phi^2\}\\
&= \{sx\in N: uvs^2x=s^2x \ \text{for all}\ u,v\in \phi\}\\
&= \{sx\in N: su'v'sx=s^2x \ \text{for all}\ u',v'\in \phi\}\\
&= \{sx\in N: u'v'sx=sx \ \text{for all}\ u',v'\in \phi\}=A.
\end{align*}

Step 4.~``$A=eN$ for some $e\in I(N)$'': Fix any $u\mydef 2f-1\in \phi$. If $s\in C(u)\cap \Inv(N)$ then we know $s\in \Inv(N)$ and $wsw=s$ for some $w\in \phi$. By property Lemma~\ref{lem2.5}\eqref{delta.cc}, $s\in N(\phi)$. Hence $sA=A$ by Step 3. Define $B\mydef \{b : 1+b\in \phi^2\}$. We get $A=\{x\in N : bx=0 \ \text{for all} \ b\in B\}$. Applying Lemma~\ref{lem2.8} to $u, B, A$ (recall $sA=A$ for all $s\in C(u)\cap \Inv(N)$) we get $A\in\{\{0\}, I^{+}(u), I^{-}(u), N\}$. By Lemma~\ref{lem2.4}, $A=I^+(u)=fN$ or $A=I^-(u)=(1-f)N$.

Step 5.~``$\phi\subseteq \Delta^+(e)$ or $\phi\subseteq \Delta^-(e)$''. Suppose $A=I^{+}(u)=fN$ with $u=2f-1$ from Step 4. Take any $v\in \phi$. By property Lemma~\ref{lem2.5}\eqref{delta.bb} there exsits $w\in \phi$ such that $wvw=u$. Since $w\in \phi$, property Lemma~\ref{lem2.5}\eqref{delta.cc} gives $w\in N(\phi)$ (because $w^\star ww^\star=w$ for some $w^\star\in \phi$). By Step 3, $wA=A$, so
\begin{align*}
I^+(v)&= \{wx\in N: vwx=wx\}=\{wx\in N : wux=wx\}\\
&=\{wx\in N : ux=x\}= wI^+(u)=wA=A.
\end{align*}
It follows that $v\in \Delta^+(f)$ from $I^+(v)=fN$. Hence $\phi\subseteq \Delta^+(f)$. Similarly, if $A=I^{-}(u)=(1-f)N$ with $u=2f-1$ from Step 4, then $\phi\subseteq \Delta^-(1-f)$.
\end{proof}

\begin{lem}\label{lem2.10}
Let $N$ be a von Neumann factor containing a nonzero idempotent e. Set $A\mydef eN$ and $\phi\mydef  \Delta^+(e)$. Then $A=I^+(\phi^2)$.
\end{lem}
\begin{proof}
If $e=1$ then  $A=N$, $\phi=\Delta^+(1)=\{2f-1 : fN=N\}=\{1\}$ and $I^+(\phi^2)=I^+(\{1\})=\{x\in N: tx=x \ \text{for all}\ t\in \{1\}\}=N$, so $A=I^+(\phi^2)$. We may assume $1-e\neq 0$. Using property Lemma~\ref{lem.vonN.2}\eqref{vonN.6} select $x\in N$ such that $ex(1-e)\neq 0$. Then $f=e+ex(1-e)\in I(N)$ and $eN=fN$, see Lemma~\ref{lem2.4}\eqref{max.3}. Since $f\neq e$, $2f-1$ and $2e-1$ are distinct elements of $\phi$.

Since $\phi$ is a $\Delta^+$-set it satisfies properties \eqref{delta.aa}-\eqref{delta.dd} of Lemma~\ref{lem2.5}. Define $A'\mydef I^+(\phi^2)$. By Lemma~\ref{lem2.9}, $A'=fN$ for some $f\in I(N)$ and $\phi\subseteq \Delta^+(f)$ or $\phi\subseteq \Delta^-(f)$. 

Suppose $\phi\subseteq \Delta^-(f)$. Select distinct $u,v\in \phi$. Using $u,v\in \phi\subseteq \Delta^-(f)$ we get $I^-(u)=I^-(v)=fN$. Using $u,v\in \phi= \Delta^+(e)$ we get $I^+(u)=I^+(v)=eN$. By Lemma~\ref{lem2.4}\eqref{max.2}, $u=v$ giving a contradiction. Consequently $\phi\subseteq \Delta^+(f)$. Select $u\in \phi$. Using $u\in \phi\subseteq \Delta^+(f)$ we get $I^+(u)=fN$. Using $u\in \phi= \Delta^+(e)$ we get $I^+(u)=eN$. We conclude 
$$A=eN=I^+(u)=fN=A'=I^+(\phi^2).\qedhere$$
\end{proof}

\begin{lem}\label{lem2.11}
Let $N$ be a von Neumann factor. Then every $\Delta^+$-set is maximal among nonempty sets $\phi\subseteq \Inv(N)$ satisfying \eqref{delta.aa}-\eqref{delta.dd} of Lemma~\ref{lem2.5}.
\end{lem}
\begin{proof}
Fix $\phi\mydef\Delta^+(e)$ for $e\in I(N)$. Then $\phi$ satisfies \eqref{delta.aa}-\eqref{delta.dd} of Lemma~\ref{lem2.5}. 

To verify maximality of $\phi$ select any $\phi'\subseteq \Inv(N)$ satisfying \eqref{delta.aa}-\eqref{delta.dd} (with $\phi'$ in place of $\phi$) of Lemma~\ref{lem2.5} and contaning $\phi$. Assume $\phi'\neq \phi$. We derive a contradiction. 

Since $\phi$ is nonempty and $\phi\subsetneq \phi'$ we get $|\phi'|>1$. Define $A'\mydef I^+(\phi'^2)$. By Lemma~\ref{lem2.9}, $A'=fN$ for some $f\in I(N)$ and $\phi'\subseteq \Delta^+(f)$ or $\phi'\subseteq \Delta^-(f)$. Suppose $\phi'\subseteq \Delta^+(f)$. Select $u\in \phi$. Using $u\in \phi\subseteq \phi' \subseteq\Delta^+(f)$ we get $I^+(u)=fN$. Using $u\in \phi= \Delta^+(e)$ we get $I^+(u)=eN$, so $fN=eN$ and $\phi'\subseteq \Delta^+(f)= \Delta^+(e)=\phi\subsetneq \phi'$. Consequently $\phi'\subseteq \Delta^-(f)$. Select $u\in \phi$. Using $u\in \phi\subseteq \phi' \subseteq\Delta^-(f)$ we get $I^-(u)=fN$. Using $u\in \phi= \Delta^+(e)$ we get $I^+(u)=eN$. We consider two cases:

Case 1. Suppose $e\neq 0$: By Lemma~\ref{lem2.10}, $eN=I^+(\phi^2)$. We also have $I^+(\phi'^2)\subseteq I^+(\phi^2)$ because if $x\in I^+(\phi'^2)$ then $tx=x$ for all $t\in \phi'^2$, hence $tx=x$ for all $t\in \phi^2$ ($\subseteq \phi'^2$), so $x\in I^+(\phi^2)$. It follows that 
$$I^-(u)=fN=A'= I^+(\phi'^2)\subseteq I^+(\phi^2)=eN=I^+(u).$$
By Lemma~\ref{lem2.4}\eqref{max.1}, $I^-(u)=\{0\}$, so $f=0$ and $\phi'\subseteq \Delta^-(f)=\Delta^-(0)=\{1\}$. But this contradicts $|\phi'|>1$.

Case 2. Suppose $e=0$: Then $I^+(u)=eN={0}$ and, using Lemma~\ref{lem2.4}\eqref{max.1}, $fN=I^-(u)=N$, so $f=1$ and $\phi'\subseteq \Delta^-(f)=\Delta^-(1)=\{-1\}$. This also contradicts $|\phi'|>1$.

We deduce that $\phi'= \phi$, so $\phi$ is maximal.
\end{proof}

\begin{prop}\label{prop2.12}
Let $N$ be a von Neumann factor. Then every $\Delta$-set is maximal among nonempty sets $\phi\subseteq \Inv(N)$ satisfying \eqref{delta.aa}-\eqref{delta.dd} of Lemma~\ref{lem2.5}.
\end{prop}
\begin{proof}
Lemma~\ref{lem2.5} and Lemma~\ref{lem2.11} provides the desired result for $\Delta^+$-sets, so we only need to consider  $\Delta^-$-sets. Notice that any $\Delta^-$-set $\phi$ satisfies the properties \eqref{delta.aa}-\eqref{delta.dd} of Lemma~\ref{lem2.5} because the properties are independent of the sign of $\phi$.

Maximality of a $\Delta^-$-set, say $\phi\mydef\Delta^-(e)$, can be verified by modifying the proof of Lemma~\ref{lem2.11}: The proof is unchanged until reaching $A'=fN$ for $f\in I(N)$ and $\phi'\subseteq \Delta^+(f)$ or $\phi'\subseteq \Delta^-(f)$. Then, if $u\in \phi\subseteq \phi' \subseteq\Delta^-(f)$ one shows $fN=I^-(u)=eN$ and $\phi'\subsetneq \phi'$. Moreover, if $u\in \phi\subseteq \phi'\subseteq \Delta^+(f)$, one shows that $fN=I^+(u)\subseteq I^-(u)$ and $f=0$ (in Case 1) or $I^-(u)=eN={0}$, $I^+(u)=fN$ and $f=1$ (in Case 2). This implies $|\phi'|\leq 1$, but $|\phi'|>1$. 
\end{proof}

\smallskip
\emph{Proof of Theorem~\ref{thmB}:} Let $\phi$ be a nonempty subset of $\Inv(N)$. If $\phi$ is a $\Delta^+$-set or a $\Delta^-$-set, then $\phi$ is a maximal set among the nonempty subset of involutions in $N$ satisfying \eqref{delta.a}-\eqref{delta.d} of Lemma~\ref{lem2.5}, see Proposition~\ref{prop2.12}.

Conversely, assume $\phi$ is a maximal set among the nonempty subset of involutions in $N$ satisfying \eqref{delta.a}-\eqref{delta.d} of Lemma~\ref{lem2.5}. We show $\phi$ is a $\Delta$-set.

Suppose $|\phi|>1$. Define $A\mydef I^+(\phi^2)$. By Lemma~\ref{lem2.9}, $A=fN$ for some $f\in I(N)$ and $\phi\subseteq \Delta^+(f)$ or $\phi\subseteq \Delta^-(f)$. By maximality of $\phi$ we get that $\phi= \Delta^+(f)$ or $\phi= \Delta^-(f)$, so $\phi$ is a $\Delta$-set.

Suppose $|\phi|=1$. If $\phi=\{1\}$ or $\phi=\{-1\}$, then $\phi$ is a $\Delta$-set because $\Delta^+(1)=\{1\}$ and $\Delta^-(1)=\{-1\}$. We claim no other option is possible. To see this assume $\phi=\{u\}$ for $u\mydef 2e-1$ such that $u\neq 1$ and $u\neq -1$. Using that $I^+(u)=eN$ (Lemma~\ref{lem2.4}\eqref{max.0}) we get that $\phi=\{u\}\subseteq \Delta^+(e)$. By maximality of $\phi$, $\{u\}=\phi=\Delta^+(e)$. Since $u$ differs from $1,-1$, both $e$ and $1-e$ are nonzero. Using property Lemma~\ref{lem.vonN.2}\eqref{vonN.6} select $x\in N$ such that $ex(1-e)\neq 0$. If follows that $f=e+ex(1-e)\in I(N)$ and $fN=eN$, see Lemma~\ref{lem2.4}\eqref{max.3}. Therefore $u\neq 2f-1\in \Delta^+(e)=\{u\}$. Contradiction.
\qed

\begin{remark}
It is unclear to which extend the characterisation in Theorem~\ref{thmB} remains valid for unital rings $N$ which are neither von Neumann factors nor regular rings. We suspect that this should be true for any unital rings $N$ satisfying properties \eqref{vonN1}-\eqref{vonN7} of Lemma~\ref{lem.vonN} and the following additional property
\begin{enumerate}[\emph{(8)}]
\item\label{vonN8} \emph{Suppose $e\in I(N)$, $B\subseteq Ne$ and $A\mydef \{x\in eN : bx=0 \ \text{for all} \ b\in B\}$ satisfy $A\neq \{0\}$ and $vA=A$ for all $v\in \Inv(eNe)$. Then $A=eN$.}
\end{enumerate}
\end{remark}

\section{Idempotents and the proof of Theorem~\ref{thmD}.}

In this section we prove Theorem~\ref{thmD}. First we establish some notation. Given two von Neumann factors $N$ and $M$, and a group isomorphism $\varphi\colon GL(N)\to GL(M)$ between their general linear groups, the formula
\begin{align*}
1-2\theta(e)=\varphi(1-2e), \ \ \ e\in I(N),
\end{align*}
induces a bijection $\theta\colon I(N)\to I(B)$ between the set of idempotents of $N$ and $M$. To simplify notation we record the following:

\begin{ntn}\label{note}
(i) The quadruple $(N,M,\varphi,\theta)$ will denote a pair of von Neumann factors $N$ and $M$, a group isomorphism $\varphi\colon GL(N) \to GL(M)$, and the induced bijection $\theta \colon I(N)\to I(M)$ where $1-2\theta(e)=\varphi(1-2e)$.

(ii) With $(N,M,\varphi,\theta)$ as above and $e\in I(N)$ we write $\Delta^+(e)\to \Delta^+$ (resp.~$\Delta^+(e)\to \Delta^-$) to indicate that $\varphi$ maps the $\Delta^+$-set $\Delta^+(e)$ in $GL(N)$ into a $\Delta^+$-set (resp.~$\Delta^-$-set) in $GL(M)$. 

(iii) We implicitly consider the quadruple $(M, N, \varphi^{-1},\theta^{-1})$. We will also write $\Delta^+(f)\to \Delta^+$ and $\Delta^+(f)\to \Delta^-$ for $f\in I(M)$ to indicate what $\varphi^{-1}$ does to the $\Delta^+$-set $\Delta^+(f)$. 
\end{ntn}

\begin{remarks}\label{rem3.2}
(i) Let $(N,M,\varphi,\theta)$ be as in (\ref{note}). By Theorem~\ref{thmB} the bijection $\varphi$ maps each $\Delta$-set into a  $\Delta$-set. In particular for $e\in I(N)$ 
\begin{align*}
\Delta^+(e)\to \Delta^+ \ \ \ \text{or} \ \ \ \Delta^+(e)\to \Delta^-.
\end{align*}
(ii) No $\Delta$-set except for $\{1\}$ and $\{-1\}$ is both a $\Delta^+$-set and $\Delta^-$-set: To see this assume
$$\phi=\Delta^+(e)=\Delta^-(f),$$ 
for some $e,f\in I(N)$. For $u\mydef 2e-1$ and $v\mydef 2f-1$ we have that $u,-v\in \phi$, so $I^-(u)=fN$ and $I^+(-v)=eN$. Using Lemma~\ref{lem2.4}\eqref{max.0}, $I^-(-v)=fN$ and $I^+(u)=eN$. Hence $u=-v$ and $e=1-f$ (see Lemma~\ref{lem2.4}\eqref{max.2}). We have $\Delta^+(e)\cap \Delta^-(1-e)=\{u\}$ (because if $w\mydef 2g-1\in \Delta^+(e)\cap \Delta^-(1-e)$ then $I^\pm(w)=I^\pm(u)$, so $w=u$), so $\phi=\{u\}$. The proof of Theorem~B (last paragraph) gives that $u$ is $1$ or $-1$, as claimed. In particular for nontrivial $e\in I(N)$
\begin{align*}
\text{either} \ \ \ \Delta^+(e)\to \Delta^+ \ \ \ \text{or} \ \ \ \Delta^+(e)\to \Delta^-.
\end{align*}
(iii)  $(N,M,\varphi,\theta)$ be as in (\ref{note}). Then $\varphi(-1)=-1$, see \cite{MR3552377}.
\end{remarks}

We now establish what is the image of each $\Delta^+$-set via the bijection $\varphi$. Using $\varphi(-1)=-1$ this also gives the image of each $\Delta^-$-set.

\begin{lem}\label{lem3.3}
Let $(N,M,\varphi,\theta)$ be as in (\ref{note}). Then for each $e\in I(N)$
\begin{align*}
\varphi (\Delta^+(e))=\left\{
\begin{array}{ll}
\Delta^+(\theta(e)),&\mbox{ if } e\neq 0,1 \mbox{ and }\Delta^+(e)\to \Delta^+\\
\Delta^-(\theta(1-e)),&\mbox{ if } e\neq 0,1 \mbox{ and }\Delta^+(e)\to \Delta^-\\
\Delta^+(1),&\mbox{ if } e =1\\
\Delta^+(0),&\mbox{ if } e =0
\end{array}\right.
\end{align*}
\end{lem}
\begin{proof}
Fix any $e\in I(N)$. Since $\varphi(-1)=-1$, $\varphi (\Delta^+(1))=\varphi(1)=1=\Delta^+(1)$ and $\varphi (\Delta^+(0))=\varphi(-1)=-1=\Delta^+(0)$. We may therefore assume that $e$ is nontrivial. By Remark~\ref{rem3.2}(ii) either $\Delta^+(e)\to \Delta^+$ or $\Delta^+(e)\to \Delta^-$. 

Suppose $\Delta^+(e)\to \Delta^+$. Select $g\in I(M)$ such that $\varphi(\Delta^+(e))=\Delta^+(g)$. By Lemma~\ref{lem2.4}\eqref{max.0}, $I^+(2e-1)=eN$, so $2e-1\in \Delta^+(e)$. Hence $\varphi(2e-1)\in \Delta^+(g)$. It follows that $2\theta(e)-1 \in \Delta^+(g)$, so $\theta(e)M=I^+(2\theta(e)-1)=gM$. We deduce $\varphi(\Delta^+(e))=\Delta^+(g)= \{v\in \Inv(M) : I^{+}(v)=gM\}=\Delta^+(\theta(e))$.

Suppose $\Delta^+(e)\to \Delta^-$. Select $g\in I(M)$ such that $\varphi(\Delta^+(e))=\Delta^-(g)$. By Lemma~\ref{lem2.4}\eqref{max.0}, $2\theta(e)-1\in \Delta^-(g)$, so $\theta(1-e)M=I^-(2\theta(e)-1)=gM$. So $\varphi(\Delta^+(e))=\Delta^-(g)= \{v\in \Inv(M) : I^{-}(v)=gM\}=\Delta^-(\theta(1-e))$.
\end{proof}

\begin{lem}\label{lem3.4}
Let $(N,M,\varphi,\theta)$ be as in (\ref{note}). Let $e\in N$ be a nontrivial idempotent and $f\mydef \theta(e)$ its image via $\theta$. Then the following holds:
\begin{enumerate}
\item\label{delta.cap} We have $\Delta^+(e)\cap \Delta^-(1-e)=\{2e-1\}$.
\item\label{delta.plus} If $\Delta^+(e)\to \Delta^+$, then $\Delta^+(f)\to \Delta^+$.
\item\label{delta.minus} If $\Delta^+(e)\to \Delta^-$, then $\Delta^+(1-e), \Delta^+(1-f), \Delta^+(f)\to \Delta^-$.
\end{enumerate}
\end{lem}
\begin{proof}
``\eqref{delta.cap}'': Fix any $w\mydef 2g-1\in\Inv(N)$. Set $u\mydef 2e-1$. Using Lemma~\ref{lem2.4}\eqref{max.0} we obtain $I^+(u)=eN$ and $I^-(u)=(1-e)N$. Now take any $w\in \Delta^+(e)\cap \Delta^-(1-e)$. Then $I^+(w)=eN$ and $I^-(u)=(1-e)N$, so $I^\pm(u)=I^\pm(w)$. By Lemma~\ref{lem2.4}\eqref{max.2}, $w=u$, so $g=e$. We conclude $\Delta^+(e)\cap \Delta^-(1-e)=\{u\}$.

``\eqref{delta.plus}'': Suppose $\Delta^+(e)\to \Delta^+$. By Lemma~\ref{lem3.3}, $\varphi(\Delta^+(e))=\Delta^+(\theta(e))$. Hence $\varphi^{-1}(\Delta^+(f))=\varphi^{-1}(\Delta^+(\theta(e))=\Delta^+(e)$, so $\Delta^+(f)\to \Delta^+$.

``\eqref{delta.minus}'': Suppose $\Delta^+(e)\to \Delta^-$. Assume that $\Delta^+(1-e)\to \Delta^+$, we derive a contradiction. By Lemma~\ref{lem3.3}, we have $\varphi(\Delta^+(e))=\Delta^-(\theta(1-e))$ and $\varphi(\Delta^+(1-e))=\Delta^+(\theta(1-e))$. Using $\varphi(-1)=-1$ if follows that
$$\varphi(\Delta^+(e))=\varphi(-1)\varphi(\Delta^+(1-e))=\varphi(\Delta^-(1-e)).$$
By \eqref{delta.cap} we get $\{2e-1\}\subseteq\Delta^+(e)\subseteq \Delta^+(e)\cap \Delta^-(1-e)=\{2e-1\}$. But we know $|\Delta^+(e)|>1$ for nontrivial $e$ (see Lemma~\ref{lem.vonN.2}\eqref{vonN.6} and Lemma~\ref{lem2.4}\eqref{max.3}). Contradiction. Therefore $\Delta^+(1-e)\to \Delta^-$, cf.~Remark~\ref{rem3.2}(ii).

Since $\Delta^+(e)\to \Delta^-$, we get that $\Delta^-(e)\to \Delta^+$ and $\Delta^-(1-e)\to \Delta^+$ (by the preceding paragraph and by $\varphi(-1)=-1$). Using that $\Delta^-(e)\to \Delta^+$, $\varphi(\Delta^-(e))=\Delta^+(\theta(1-e))=\Delta^+(1-f)$, so $\Delta^+(1-f)\to \Delta^-$. Using that $\Delta^-(1-e)\to \Delta^+$, $\varphi(\Delta^-(1-e))=\Delta^+(\theta(1-(1-e)))=\Delta^+(f)$, so $\Delta^+(f)\to \Delta^-$.
\end{proof}

Recall we write $e\eqv_N f$ (or $e\eqv f$)  whenever $e, f\in I(N)$ and $eN=fN$. This is an equivalence relation on $I(N)$. Let $[e]$ denote the equivalence class containing $e$ and $I(N)/\smallspace\eqv$ the set of equivalence classes, i.e.,
\begin{align*}
[e]\mydef\{f \in I(N): f\eqv e\}, \ \ \ I(N)/\smallspace\eqv\ \mydef \{[e] : e\in I(N)\}. 
\end{align*}

\smallskip
\emph{Proof of Theorem~\ref{thmD}:} Let $(N,M,\varphi,\theta)$ be as in (\ref{note}). Define the map $\tilde\theta\colon I(N)/\smallspace\eqv \ \to I(M)/\smallspace\eqv$ as follows:
\begin{align*}
\tilde\theta ([e])=\left\{
\begin{array}{ll}
{[\theta (e)]},&\mbox{ if } e\neq 0,1 \mbox{ and }\Delta^+(e)\to \Delta^+\\
{[\theta (1-e)]},&\mbox{ if } e\neq 0,1 \mbox{ and }\Delta^+(e)\to \Delta^-\\
{[1]},&\mbox{ if } e =1\\
{[0]},&\mbox{ if } e =0
\end{array}\right.
\end{align*}
We show $\tilde\theta$ is well defined. Fix any equivalence class $\phi$ of idempotents generating the same right ideal.

Suppose $|\phi|=1$. By Lemma~\ref{lem.vonN.2}\eqref{vonN.6} and Lemma~\ref{lem2.4}\eqref{max.3}, $\phi=[1]$ or $\phi=[0]$, so there is nothing to prove (as $\phi$ has precisely one representative).

Suppose $|\phi|>1$. Select any $e,f\in \phi$.  We have $[e]=[f]$. Recall that (by definition) $[e]=[f]$ iff $\Delta^+(e)=\Delta^+(f)$. Suppose $\Delta^+(e)\to \Delta^+$. Lemma~\ref{lem3.3} ensures that $\Delta^+(\theta(e))=\Delta^+(\theta(f))$, so $[\theta(e)]=[\theta(f)]$. Consequently, 
\begin{align*}
\tilde\theta([e])= [\theta(e)]=[\theta(f)]=\tilde\theta([f])
\end{align*}
Suppose $\Delta^+(e)\to \Delta^-$. By Lemma~\ref{lem3.3}, $[\theta(1-e)]=[\theta(1-f)]$, hence $\tilde\theta([e])= [\theta(1-e)]=[\theta(1-f)]=\tilde\theta([f])$. So the map $\tilde\theta$ is well-defined.

We show $\tilde\theta$ is a bijection. Define the map $\tilde\psi\colon I(M)/\smallspace\eqv \ \to I(N)/\smallspace\eqv$ by
\begin{align*}
\tilde\psi ([f])=\left\{
\begin{array}{ll}
{[\theta^{-1} (f)]},&\mbox{ if } f\neq 0,1 \mbox{ and }\Delta^+(f)\to \Delta^+\\
{[\theta^{-1} (1-f)]},&\mbox{ if } f\neq 0,1 \mbox{ and }\Delta^+(f)\to \Delta^-\\
{[1]},&\mbox{ if } f =1\\
{[0]},&\mbox{ if } f =0
\end{array}\right.
\end{align*}
The map is well-defined by arguments analogues to those establishing well-definiteness of $\tilde\theta$. We show $\tilde\psi$ is the inverse for $\tilde\theta$. Fix a nontrivial idempotent $e\in N$ and set $f\mydef\theta(e)$.

Suppose $\Delta^+(e)\to \Delta^+$. Then $\Delta^+(f)\to \Delta^+$ (by Lemma~\ref{lem3.4}). By definition of $\tilde\theta$ and $\tilde\psi$, $\tilde\theta([e])= [\theta(e)]$ (so $\tilde\theta([e])=[f]$) and $\tilde\psi([f])= [\theta^{-1}(f)]$ (so $\tilde\psi([f])=[e]$). We conclude $$\tilde\psi\circ\tilde\theta([e])=\tilde\psi([f])=[e], \ \ \ \tilde\theta\circ\tilde\psi([f])=\tilde\theta([e])=[f].$$

Suppose $\Delta^+(e)\to \Delta^-$. Using Lemma~\ref{lem3.4}, $\Delta^+(1-f)\to \Delta^-$. We get that $\tilde\theta([e])= [\theta(1-e)]$ (so $\tilde\theta([e])=[1-f]$) and $\tilde\psi([1-f])= [\theta^{-1}(1-(1-f))]$ (so $\tilde\psi([1-f])=[e]$). Consequently $$\tilde\psi\circ\tilde\theta([e])=\tilde\psi([1-f])=[e].$$
Knowing that $\Delta^+(e)\to \Delta^-$, we also get $\Delta^+(1-e)\to \Delta^-$ and $\Delta^+(f)\to \Delta^-$ (by Lemma~\ref{lem3.4}). Hence we both have that $\tilde\theta([1-e])= [\theta(1-(1-e))]$ (so $\tilde\theta([1-e])=[f]$) and $\tilde\psi([f])= [\theta^{-1}(1-f)]$ (so $\tilde\psi([f])=[1-e]$). Therefore 
$$\tilde\theta\circ\tilde\psi([f])=\tilde\theta([1-e])=[f].$$
We conclude that $\tilde\theta$ has an inverse, hence it is a bijection.
\qed

\section*{Acknowledgements}
Part of this research was conducted while the authors were participating in the research program \emph{Classification of operator algebras: complexity, rigidity, and dynamics} at the Mittag-Leffler Institute, January--April 2016, and the Intensive Research Program \emph{Operator algebras: dynamics and interactions} at the Centre de Recerca Matem{\`{a}}tica, March--July 2017.  This research was supported by EIS Distinguished Visitors Program, by Australian Research Council grant DP150101598 and by a grant from NSERC Canada.

\providecommand{\bysame}{\leavevmode\hbox to3em{\hrulefill}\thinspace}
\providecommand{\MR}{\relax\ifhmode\unskip\space\fi MR }
\providecommand{\MRhref}[2]{%
  \href{http://www.ams.org/mathscinet-getitem?mr=#1}{#2}
}
\providecommand{\href}[2]{#2}

\end{document}